\documentclass{amsart}
\usepackage[utf8]{inputenc}
\usepackage{comment}

\title{CoveringDihedral}
\author{}
\date{\today}

\usepackage[margin=1in]{geometry}

\newcommand{\Aut}{\mathrm{Aut}}

      \title[On Covers of Dihedral 2-Groups by Powerful Subgroups]
            {On Covers of Dihedral 2-Groups by Powerful Subgroups}
\author[R.~Atanasov]{Risto~Atanasov}
\address{Department of Mathematics and Computer Science \\
Stillwell 426 \\
Western Carolina University \\
Cullowhee, NC 28723 USA}
\email{ratanasov@email.wcu.edu}

\author[A.~Gregory]{Adam ~Gregory}
\address{Department of Mathematics and Computer Science \\
Stillwell 426 \\
Western Carolina University \\
Cullowhee, NC 28723 USA}
\email{gregoryadam9@gmail.com}

\author[L.~Guatelli]{Luke Guatelli}
\address{Department of Mathematics and Computer Science \\
Stillwell 426 \\
Western Carolina University \\
Cullowhee, NC 28723 USA}
\email{lrguatelli1@catamount.wcu.edu}

\author[A.~Penland]{Andrew~Penland}
\address{Department of Mathematics and Computer Science \\
Stillwell 426 \\
Western Carolina University \\
Cullowhee, NC 28723 USA}
\email{adpenland@email.wcu.edu}


\usepackage{ifpdf}
\ifpdf
  \usepackage[pdftex,pagebackref,hyperindex,
   pdftitle={DihedralPowerfulCoveringNumber},
  pdfauthor={},
 pdfsubject={2010 Mathematics Subject Classification:},
pdfkeywords={},
  ]{hyperref}
\else
  \usepackage[hypertex,pagebackref,hyperindex]{hyperref}
\fi

\usepackage{geometry}
\usepackage{amsmath,amssymb,mathrsfs, euscript}
\usepackage{amsthm}
\usepackage{amsfonts}
\usepackage[all]{xypic}
\usepackage[pdftex]{graphicx}

\numberwithin{equation}{section}

\newtheorem{theorem}[equation]{Theorem}

\newtheorem{lemma}[equation]{Lemma}

\newtheorem{proposition}[equation]{Proposition}
\newtheorem{corollary}[equation]{Corollary}

\theoremstyle{definition}

\newtheorem{defn}[equation]{Definition}
\newtheorem{example}[equation]{Example}
\newtheorem*{remark}{Remark}

\begin{document}

\maketitle

\begin{abstract}
A finite $p$-group $G$ is called \textit{powerful} if either $p$ is odd and $[G,G]\subseteq G^p$ or $p=2$ and $[G,G]\subseteq G^4$. A {\em{cover}} for a group is a collection of subgroups whose union is equal to the entire group. We will discuss covers of $p$-groups by powerful $p$-subgroups. The size of the smallest cover of a $p$-group by powerful $p$-subgroups is called the \textit{powerful covering number}. In this paper we determine the powerful covering number of the dihedral 2-groups.
\end{abstract}

\maketitle

\section{Introduction}
If $G$ is a group, a \textit{cover} of $G$ is a collection of proper subgroups whose union is equal to $G$. A cover of a group $G$ is \textit{minimal} if there is no cover with a smaller number of subgroups. The \textit{covering number} of $G$ is the size of a minimal cover.  
 
Interest in covers and covering numbers goes back to  G.A. Miller, who considered covers by subgroups with pairwise trivial intersection~\cite{Mi}. Such a cover is known as a \textit{partition}. The classification of finite groups admitting a partition is expressed in the following theorem, which combines results primarily due to Baer, Kegel and Suzuki (see ~\cite{Ba, Ke, Su}).
\begin{theorem} [The Classification Theorem]  $G$ is a finite group admitting a nontrivial partition if and only if $G$ is isomorphic to one of the following groups:
\begin{enumerate}
\item $S_4$;
\item  a $p$-group with $H_p(G) \ne G$, where $H_p(G)=\langle  x\in G \mid x^p\ne1\rangle$ ;
\item  a group of Hughes-Thompson type;
\item  a Frobenius group;
\item $PSL(2, p^n)$ with $p^n\ge 4$;
\item $PGL(2, p^n)$ with $p^n\ge 5$  and $p$ odd;
\item $Sz(2^{2n+1})$,
\end{enumerate}
where $p$ is a prime and $n$ is a natural number.\end{theorem}

Isaacs~\cite{Isaacs} completely classified all finite groups with partitions by subgroups all having the same order. The survey by Zappa~\cite{Zappa} offers many additional results related to partitions and their generalizations. 

Many authors have considered the covering number for various families of finite groups.  Scorza~\cite{Scorza} first studied groups that could be covered by three proper subgroups.  Cohn \cite{Cohn} introduced the notion of a covering number and examined some of its properties. Tomkinson \cite{Tomkinson} posed the question of which integers can be the covering numbers of a group and gave a formula for the covering number of any solvable group.  Bryce, Fedri, and Serena~\cite{BryceFedriSerena} found the covering numbers of many finite linear groups. Lucido~\cite{Lucido} gave the covering numbers of the Suzuki groups. Holmes \cite{Holmes} found exact values and bounds for the covering numbers of sporadic simple groups. Garonzi~\cite{AtMost25} has determined the groups that can be covered by 25 or fewer subgroups. The exceptional groups of Lie type with covering number equal to $n$ were explored by Lucido~\cite{Lucido2}.  There have also been many results on the covering number of symmetric and alternating groups. Mar\'{o}ti \cite{Maroti} considered the asymptotic properties of these covering numbers. Kappe and Redden~\cite{KappeRedden} gave exact values and bounds for the covering numbers of some small alternating and symmetric groups. Swarz~\cite{Swarz} gave exact values for $\sigma(S_n)$ in the case when $n$ is divisible by 6. Kappe, Nikolova-Popova, and Swarz~\cite{KNPS} used a novel application of techniques from integer linear programming to give the exact covering number of some symmetric groups.
The survey paper by Serena \cite{Se} offers a comprehensive survey of known results on group covers at the time it was written.  

Mathematicians have also considered generalizations of group covers and covering numbers, placing additional requirements either on the cover or on the subgroups it contains.  For instance, Bryce and Serena~\cite{BryceSerena} classified the groups that have a minimal cover by abelian subgroups.  Bhargava~\cite{Bhargava} determined the groups with covers consisting only of normal subgroups.  Jabara and Lucido~\cite{JabaraLucido} considered the groups with coverings by Hall subgroups. Garonzi and Lucchini \cite{Lucchini} studied what they called \textit{normal covers} - group covers which are closed under conjugation by any element of the group. They also considered the number of distinct conjugacy classes that could appear in a normal cover.  Foguel and Ragland~\cite{FoguelRagland} investigated covers by abelian groups, all of which are pairwise isomorphic. In~\cite{AFP}, Atanasov, Foguel, and Penland considered covers by subgroups that all have the same order and have mutually isomorphic pairwise intersections. These covers were a generalization of both equal partitions and strict $S$-partitions.

In this paper we study covers of finite $p$-groups by powerful subgroups. If $p$ is a prime number, a \textit{finite $p$-group} is a group of order $p^{\alpha}$ for some nonnegative integer $\alpha$. A finite $p$-group $G$ is called \textit{powerful} if either $p$ is odd and $[G,G]\subseteq G^p$ or $p=2$ and $[G,G]\subseteq G^4$. Powerful groups were first introduced by Lubotzky and Mann~\cite{Lubotzky1}, who later used them to provide a characterization of $p$-adic analytic pro-$p$ groups~\cite{Lubotzky2}.  They  also served an important role in the classification of finite $p$-groups by a property known as \textit{coclass}, introduced by Leedham-Green and Newman in~\cite{Leedham1}(see~\cite{Leedham2} for an overview of the \textit{coclass conjectures} and their proofs). Powerful $p$-groups are also used in the study of automorphisms of $p$-groups~\cite{Khukhro}.

If $G$ is a $p$-group, we define a \textit{powerful cover} of $G$ to be a cover of $G$ by powerful subgroups. We define the \textit{powerful covering number} of $G$ to be the minimal number of subgroups in any powerful cover of $G$. In effect, the powerful covering number gives an idea of how large the powerful subgroups of a given group are.  

Powerful covering numbers do not share all of the same properties as ordinary covering numbers. We will explore some of these differences in Section~\ref{s:powerful-covering-properties}, where we prove some results on how powerful covering numbers behave with respect to homomorphic images and direct products. In Section~\ref{s:dihedral-powerful-covering-numbers}, we explicitly calculate the powerful covering numbers of the dihedral $2$-groups. We conclude with a discussion of conjectures and questions for future work. 

\section{Background} 

In the interest of making this paper self-contained and accessible to non-experts, we will make very few assumptions about the reader's knowledge of group theory. This section reviews most definitions and well-known facts necessary to establish our results. Most of this material is standard and can be found in texts such as~\cite{Gallian} or~\cite{Dummit-Foote}.  All groups in this paper are assumed to be finite. 

\subsection{Group Theory}

Let $G$ be a group. We write $e$ for the identity element of $G$. If $S$ is a subset of $G$, the \textit{group generated by $S$} is the smallest subgroup of $G$ that contains $S$ and is denoted $\langle S \rangle$. If $T$ is a set, we write $|T|$ for the cardinality of $T$. A subgroup generated by a single element $g \in G$ is denoted $\langle g \rangle$ in an abuse of notation. Such a subgroup is called \textit{cyclic}. For $g \in G$, the \textit{order} of $g$ is defined to be $|\langle g \rangle|$. An abelian group is called \textit{elementary abelian} if all nontrivial elements have order $p$.

For $g,h \in G$, the \textit{commutator of $g$ and $h$} is defined as $g^{-1}h^{-1}gh$ and denoted $[g,h]$. If $H$ and $K$ are subgroups of a group $G$, the group $[H,K]$, is the group generated by the set $\{ [h,k] \mid h \in H, k \in K \}$. For a positive integer $k$, the subgroup $G^k$ is the group generated by the set $\{ g^k \mid g \in G \}$. For a group $G$, the \textit{lower central series} is defined as $G_0 = G$, $G_i = [G_{i-1}, G]$ for $i \geq 1$. If there exists an $i$ such that $G_i$ is trivial, then $G$ is \textit{nilpotent}, and the \textit{nilpotence class of G} is defined as the least $i$ such that $G_i$ is trivial. 

A subgroup $N$ of a group $G$ is \textit{normal} if the group $\{ g^{-1} n g \mid n \in N, g \in G \}$ is equal to $N$. The \textit{normal closure} of a subgroup $H \subseteq G$ is the smallest normal subgroup of $G$ that contains $H$. We write $H^G$ for the normal closure of $H$ in $G$. It is not hard to see that a group is normal if and only if it is equal to its normal closure. 

\begin{proposition}\label{p:normal-closure-facts}
Let $G$ be a group. 
\begin{itemize}
    \item[(i.)] If $X$ is a generating set for $G$. Then $[G,G] = \langle \{ [x,y] \mid x,y \in X \} \rangle^G$.
    \item[(ii.)] If $H$ and $K$ are subgroups of $G$ with $H \leq K$, then $H^G \leq K^G$.
    \item[(iii.)] For any positive integer $k$, $G^k$ is a normal subgroup of $G$
    \end{itemize}
\end{proposition}

The \textit{center} of a group $G$, denoted $Z(G)$, is the set of elements $g \in G$ such that $[g,h] = e$ for all $h \in G$. 

A subgroup $K$ of $G$ is \textit{maximal} if whenever there exists a subgroup $H \subseteq G$ such that $K$ is a proper subgroup of $H$, it follows that $H = G$. If $H$ is any subgroup of $G$, the \textit{index} of $H$ in $G$ is equal to $\displaystyle\frac{|G|}{|H|}.$ For a $p$-group, it is known that all maximal subgroups have index $p$. If $G$ and $T$ are groups, a \textit{group homomorphism} is a map $\alpha: G \rightarrow T$ such that $\alpha(gh) = \alpha(g)\alpha(h)$ for all $g, h \in G$. The \textit{kernel} of a group homomorphism $\alpha: G \rightarrow T$ is $\text{ker }\alpha=\{g\in G|\alpha(g)=e_T\}$, where $e_T$ is the identity in $T$. Since $G$ is finite, the \textit{index} of a subgroup $H$ of $G$ is defined to be the number $\frac{|G|}{|H|}$. If $\alpha$ is a group homomorphism, then by The First Isomorphism Theorem the index of $\text{ker } \alpha$ in $G$ is equal to the cardinality of the image of $\alpha$. If $G$ and $T$ are groups and $\alpha : G \rightarrow T$ is a bijective group homomorphism, then $\alpha$ is an \textit{isomorphism} and we say that $G$ and $T$ are \textit{isomorphic} and write $G \cong T$. If $\alpha$ is an isomorphism from a group $G$ into itself, then $\alpha$ is an \textit{automorphism} of $G$. The set of all automorphisms of a group $G$ form a group under function composition, denoted by Aut$(G)$. 

A \textit{left group action} of a group $G$ on a nonempty set $A$ is a map from $G \times A$ to $A$, denoted by $g \cdot a$ for all $a \in A$ and $g \in G$, such that $g_1 \cdot (g_2 \cdot a) = (g_1 g_2) \cdot a$ and $e \cdot a = a$ for all $g_1, g_2 \in G$ and $a \in A$. We can define a right action similarly. 

If $G$ and $H$ are groups, the \textit{direct product} $G \times H$ is a group with the set $\{ (g,h) \mid g \in G, h \in H \} $ and group operation given componentwise. If $G$ and $H$ are groups with a group homomorphism $\alpha : H \rightarrow \;  \text{Aut}(G)$, the \textit{semidirect product} of $G$ and $H$, denoted $G \rtimes_\alpha H$, is the set of ordered pairs $(g,h)$ with $g \in G$ and $h \in H$ with group operation defined $(g_1, h_1)(g_2, h_2) = (g_1 g_2^{h_1}, h_1 h_2)$, where $g_2^{h_1}$ is defined as $\alpha(h_1)(g_2)$.

\subsection{Cyclic, Dihedral, and other Finite Groups}

If $t$ is a natural number, a \textit{cyclic group of order t} is the group generated by a single element of order $t$. We write $C_t$ for the cyclic group of order $t$ generated by an element $z$ of order $t$. We let $\tau$ represent the generator of $C_2$. 

For $n \geq 2$, we define $D(2^n)$ to be the set of isometries of a regular $2^n$-gon. The group $D(2^n)$ has $2^{n+1}$ elements. 

Several facts about the elements of the dihedral groups are well-known from Euclidean geometry, see e.g.~\cite[Section 2.2]{FiniteReflection} or ~\cite[Section 3.3]{CarneGeometry}. 

\begin{theorem}
Let $D(2^n)$ be the group of isometries of a regular $2^n$-gon. The following hold: 
\begin{enumerate}
    \item[(i.)] Every element of $D(2^n)$ is either a reflection or a rotation.
    \item[(ii.)] The composition of two reflections gives a rotation.
    \item[(iii.)] The composition of a reflection and a rotation is a reflection. 
\end{enumerate}
\end{theorem}

\begin{defn}
We say a group is a  \textit{dihedral $2$-group} if it is isomorphic to $D(2^n)$ for some $n \geq 2$ or if it is isomorphic to $C_2 \times C_2$. \\
\end{defn}

The following two facts about the generators of dihedral $2$-groups can be found in~\cite{FiniteReflection}.

\begin{theorem}\label{l:order-two-dihedral}
If $G$ is a $2$-group generated by two elements of order two, then $G$ is isomorphic to a dihedral $2$-group. 
\end{theorem}

\begin{proof}
See~\cite[Theorem 26.5]{Gallian}.
\end{proof}

\begin{corollary}\label{c:C2-C2-fact}
If $G$ is an abelian $2$-group generated by two elements of order two, then $G$ is isomorphic to $C_2 \times C_2$.
\end{corollary}

We write $S_n$ for the symmetric group of all permutations on $n$ elements, and $A_n$ for the subgroup of $S_n$ consisting of even permutations. 

\subsection{Powerful Groups}

As we stated in the Introduction, a finite $p$-group $G$ is called \textit{powerful} if either $p$ is odd and $[G,G]\subseteq G^p$ or $p=2$ and $[G,G]\subseteq G^4$. If a $p$-group $G$ is abelian, then $G$ is powerful since $[G,G]$ is trivial. As noted in~\cite{Lubotzky1}, the homomorphic image of a powerful group is powerful, and the direct product of two powerful groups is powerful. 

To find examples of nonabelian powerful groups, we used the computer algebra system \texttt{GAP} (Groups, Algorithms, and Programming) \cite{GAP}. Here are a few examples of nonabelian powerful groups. 

\begin{example}
Let $C_8$ be the cyclic group of order 8, generated by an element $z$. Let $\phi: C_2 \rightarrow \Aut(C_8)$ be a homomorphism which sends $\tau$ to the automorphism of $C_8$ sending $z$ to $z^5$. Let $G$ be equal to the semidirect product $C_8 \rtimes_{\phi} C_2$. By an abuse of notation, we will write $\tau$ for $(e,\tau)$ and $z$ for $(z,e)$. Notice that $G$ is generated by $\tau$ and $z$. We will show that $[G,G] \subseteq G^4$.  Since  $z^{\tau} = z^5$ and $z^{-1} = z^7$, we calculate 
\begin{align*}
[\tau, z] &= (\tau \tau^{-1}, z^{\tau} z^7) \\
&= (e, z^5z^7) \\
&= (e, z^{12}) \\
&= (e, z^3)^4.
\end{align*}
It follows that $[\tau, z] \in G^4$. Hence $ \langle [z, \tau] \rangle \subseteq G^4$, so $[G,G] \subseteq G^4$ by Proposition~\ref{p:normal-closure-facts}. 
\end{example}

\begin{example}
The previous example can be generalized. Let $C_{2^n}$ be a cyclic group of order $2^n$ generated by an element $z$ of order $2^n$, and let $\phi: C_2 \rightarrow \Aut(C_{2^n})$ be the homomorphism that sends $\tau$ to the automorphism $z \mapsto z^{2^{n-1} + 1}$. By an identical calculation as the previous example, we see that $$[z, \tau] = (z^{2^{n-1} + 2^n}, e) = ((z^{2^{n-3} + 2^{n-2}})^4, e),$$ establishing that $[G,G] \subseteq G^4$. It follows that $C_{2^n} \rtimes_{\phi} C_2$ is powerful for all $n \geq 3$.
\end{example}

\subsection{Covers and Covering Numbers}

A \textit{cover} of a $p$-group $G$ is a collection of proper subgroups whose union is equal to $G$. A cover of a group $G$ is \textit{minimal} if there is no cover with a smaller number of subgroups. The \textit{covering number} of $G$ is the size of a minimal cover. Following Cohn~\cite{Cohn}, we write $\sigma(G)$ for the covering number of $G$. Here we note some known facts about the covering number. 

\begin{proposition}[Lemma 2, ~\cite{Cohn}]\label{t:homomorphic-cover-numbers}
Let $G$ be a noncyclic group. If $H$ is a homomorphic image of $G$, then $\sigma(G) \leq \sigma(H).$
\end{proposition}

\begin{corollary}[in ~\cite{Cohn}]\label{c:direct-product-covering-number}
If $H$ and $K$ are non-cyclic groups, then $\sigma(H \times K) \leq \min(\sigma(H), \sigma(K)).$ 
\end{corollary}

\begin{theorem}[Theorem 2, ~\cite{Cohn}]\label{t:covering-number-p-groups}
If $G$ is a noncyclic $p$-group, then $\sigma(G) = p+1$.
\end{theorem}

In~\cite{BryceSerena}, Bryce and Serena give a characterization of groups with minimal covers by abelian subgroups. In order to state the result, we need a few definitions. A group $H$ is \textit{monolithic} if it has a unique nontrivial normal subroup that is contained in every other nontrivial normal subgroup; this unique normal subgroup is called the \textit{monolith} of $H$. A subgroup $K$ of a group $H$ is \textit{central} if it is contained in $Z(H)$. If $K$ is a subgroup of a group $H$, a \textit{complement} of $K$ is a subgroup $L \leq H$ such that $L \cap K$ is trivial and $H = LK$. 

\begin{theorem}[~\cite{BryceSerena}]\label{t:minimal-covers-by-abelian}
Let $G$ be a nonabelian group. Then $G$ has a minimal cover consisting of only abelian groups if and only if the factor group $G/Z(G)$ is either
\begin{enumerate}
    \item monolithic, with a non-central, elementary abelian monolith $K/Z$ having order $p^a$ having cyclic complements and with $K$ abelian, or
    \item elementary abelian of order $p^2$ for some prime number $p$,
\end{enumerate}
and for each prime number $q < p^a$ or $q < p$, as the case may be, every finite factor group of $G$ has cyclic Sylow $q$-subgroups.
\end{theorem}

One might hope that if $H$ if a subgroup of another group $G$, then $\sigma(H) \leq \sigma(G)$ would hold. This is not always true. Kappe, Nikolova-Popova, and Swarz calculated that $\sigma(S_8) = 64$~\cite[Theorem 2.1]{KNPS}, while Kappe and Redden showed that $\sigma(A_8) = 71$~\cite{KappeRedden}. 

\section{The Powerful and Abelian Covering Numbers for p-Groups}\label{s:powerful-covering-properties}

We define a \textit{powerful cover} of $G$ as a cover consisting of only powerful subgroups. An \textit{abelian cover} of $G$ is a cover consisting of only abelian subgroups.  We define the \textit{powerful covering number of $G$} as the size of a minimal powerful cover, and the \textit{abelian covering number of $G$} as the size of a minimal abelian cover of $G$. As before, we write $\sigma(G)$ for the covering number of $G$. We write $\sigma_P(G)$ for the powerful covering number, and $\sigma_A(G)$ for the abelian covering number. It should be noted that cyclic groups do not have covers by proper subgroups. 

It is natural to ask if $\sigma_P$ and $\sigma_A$ share analagous properties to those of $\sigma$. As noted in ~\cite{Lubotzky1}, neither subgroups nor homomorphic pre-images of a powerful group are necessarily powerful, and this means that certain generalizations do not come easily. We now turn our attention to properties of $\sigma_P$ and $\sigma_A$.

\begin{proposition}\label{p:all-p-groups-have-powerful-covers}
Let $G$ be a finite noncyclic $p$-group. Then $G$ has a powerful cover.
\end{proposition}

\begin{proof}
The group $G$ has a cover by the collection of all cyclic subgroups, and cyclic groups are powerful. 
\end{proof}

\begin{lemma}\label{l:powerful-abelian-covering-relationship}
If $G$ is a finite noncyclic $p$-group, the relationship
\[
\sigma(G) \leq \sigma_{P}(G) \leq \sigma_{A}(G)
\]
always holds. 
\end{lemma}

\begin{proof}
This follows immediately from the fact that abelian groups are powerful. 
\end{proof}

\begin{remark}
The Theorem of Bryce and Serena~\cite{BryceSerena} given in Theorem~\ref{t:minimal-covers-by-abelian} above completely characterizes the groups $G$ for which $\sigma(G) = \sigma_P(G) = \sigma_A(G)$. 
\end{remark}

\begin{theorem}\label{t:not-powerful-case-covering-reduces}
Let $G$ be a finite noncyclic $p$-group. If $K$ is a homomorphic image of $G$ such that $K$ is not powerful, then $\sigma_P(K) \leq \sigma_P(G)$.
\end{theorem}

\begin{proof}
 Suppose $G = \bigcup_{i=1}^n H_i$, where each $H_i$ is a proper powerful subgroup. Let $\alpha: G \rightarrow K$ be a surjective homomorphism. We see that each $\alpha(H_i)$ is also a powerful subgroup as the homomorphic image of a powerful subgroup, and that \[
 \bigcup_{i=1}^n \alpha\left(H_i\right) = K.
 \]
 It is important to check that each $\alpha(H_i
 )$ is a proper subgroup, but this ensured since each $\alpha(H_i)$ is powerful and $K$ is not. 
\end{proof}

\begin{theorem}\label{t:case-where-sigma-p-preserved-under-direct-product}
If $G$ is a finite noncyclic $p$-group such that $G$ is not powerful,  and $K$ is a powerful finite $p$-group, then  $\sigma_P(G \times K) = \sigma_P(G).$
\end{theorem}

\begin{proof}
Suppose that $G$ has a minimal covering by powerful subgroups $H_1, \ldots, H_n$. In $G \times K$, each subgroup $H_i \times K$  is powerful, since the direct products of powerful groups is powerful. Then we notice
\[
G \times K = \left( \bigcup_{i=1}^n H_i \right) \times K = \bigcup_{i=1}^n \left( H_i \times K \right),  
\]
which establishes that $\sigma_P(G \times K) \leq \sigma(G)$.  To see that $\sigma_P(G \times K) \geq \sigma_P(G)$, we note that $G$ is a non-powerful homomorphic image of $G \times K$, and apply Theorem~\ref{t:not-powerful-case-covering-reduces}. 
\end{proof}

As we will see in Section~\ref{s:dihedral-powerful-covering-numbers}, the the analog of Theorem~\ref{t:homomorphic-cover-numbers} does not hold in general for $\sigma_P$. However, there is a special case where it does. 

\begin{theorem}\label{t:special-case-powerful-preimage}
Let $G$ and $K$ be finite noncyclic $p$-groups. Suppose $H_1, \ldots, H_q$ is a minimal powerful cover of $K$. If there exists a surjective homomorphism $\alpha: G \rightarrow K$ such that $\alpha^{-1}(H_i)$ is powerful for each $i = 1, \ldots, q$, then $\sigma_P(G) \leq \sigma_P(K)$.
\end{theorem}

\begin{proof}
Let $H_1, \ldots, H_q$ be a minimal powerful cover of $K$. If $\alpha^{-1}(H_i)$ is powerful for each $i = 1, \ldots, q$, then $\{ \alpha^{-1}(H_i) \}_{i=1}^q$ forms a powerful cover of $G$.  
\end{proof}

We are also able to adapt the result of  Corollary~\ref{c:direct-product-covering-number}, but with a different proof. 

\begin{theorem}\label{t:direct-product-is-nice}
Let $G$ and $H$ be finite noncyclic powerful $p$-groups. Then $\sigma_P(G \times H) \leq \min \left(\sigma_P(G), \sigma_P(H) \right)$. 
\end{theorem}

\begin{proof}
Suppose that $K_1, \ldots, K_m$ forms a minimal powerful cover of $G$, and $L_1, \ldots, L_n$ forms a minimal powerful cover of $H$. Then the collections $\{G \times K_i \}_{i=1}^m$ and $\{H \times L_j \}_{j=1}^n$ both form powerful covers of $G \times H$, so $\sigma_P(G \times H) \leq \min \left( \sigma_P(G), \sigma_P(H) \right)$.
\end{proof}

\section{Powerful Covering Number of Dihedral 2-Groups}\label{s:dihedral-powerful-covering-numbers}

The purpose of this section is to establish the powerful covering numbers of the dihedral $2$-groups. First, we establish several facts regarding the subgroup structure of dihedral $2$-groups. All of these facts are well-known (see \cite{CarneGeometry} or \cite{ConradDihedral}), but we will provide self-contained proofs where convenient. Then, we use the properties to prove the main results on powerful covers of dihedral $2$-groups.

\subsection{Elements and Subgroups of Dihedral 2-Groups}

As previously noted, we define the group $D(2^n)$ as the set of isometries of a regular $2^n$-gon in the Cartesian plane. For definiteness and ease of calculation, let us take the vertices of this $2^n$-gon are at the points $\left(\cos\left(\displaystyle\frac{2\pi k}{2^n} \right), \sin \left( \displaystyle\frac{2 \pi  k}{2^n} \right) \right)$ for $i = 0, 1, \ldots,  2^n - 1$. We think of the elements of $D(2^n)$ as functions, with a left action on the points of the $2^n$-gon by function composition. 

We will distinguish two elements of $D(2^n)$ as a generating set. Let $a$  be the reflection across the line that makes an angle of $\displaystyle\frac{\pi}{2^n}$ with the $x$-axis, and let $b$ be the reflection of the $2^n$-gon about the $x$-axis.  We will re-use these symbols $a$ and $b$ regardless of the value of $n$; the context should always make the meaning clear. We will also write $e$ for the identity of any group, again relying on context to make clear which group. 

Since a reflection is always its own inverse, we see that in $D(2^n)$, $a^2 = b^2 = e$. A geometric calculation also establishes that $ab$ is equal to a counterclockwise rotation by $\displaystyle\frac{2 \pi}{2^n}$ radians. The rotation $ab$ will play a crucial role in our future calculations, so we record some of its properties now. 

\begin{lemma}\label{l:ab-order-lemma}
Let $n \geq 2$. The element $ab \in D(2^n)$ has order $2^{n}$, and for any $r \in \mathbb{Z}$,  $(ab)^{r} = (ba)^{2^n - r}$. 
\end{lemma}

\begin{proof}
The order of $ab$ follows directly from the fact that $ab$ represents a counterclockwise rotation of $\displaystyle\frac{2\pi}{2^n}$ radians. For the second statement, notice that since $a$ and $b$ both have order two, $(ab)(ba) = e$. Thus we know that $ba = (ab)^{-1}$ and hence represents a clockwise rotation of  $-\displaystyle\frac{2\pi}{2^n}$ radians. Thus, $(ab)^r$ represents a counterclockwise rotation of $\displaystyle\frac{r 2 \pi}{2^n}$, which gives the same transformation of the clockwise rotation of $\displaystyle\frac{(2^n - r)2 \pi}{2^n}$ radians represented by $(ba)^{2^n - r}$.
 \end{proof}
 
\begin{lemma}\label{l:normal-form}
Suppose $g \in D(2^n)$. 
\begin{enumerate}
\item[(i.)] The element $g$ can be written uniquely as $g = (ab)^ja^k$ for some integers $j$ and $k$ with $0 \leq j < 2^{n}$ and $0 \leq k \leq 1$. 
\item[(ii.)] If $g = (ab)^ja$ for some $j$ with $0 \leq j < 2^{n}$, then the order of $g$ has order two. 
\item[(iii.)] If $g = (ab)^j$ for some $j$ with $0 \leq j < 2^{n}$, then the order of $g$ is  $$\displaystyle\frac{2^{n-1}}{\gcd(j,2^{n-1})}$$. 
\item[(iv.)] For any $g \in D(2^n)$, $g (ab)^{2^{n-1}} = (ab)^{2^{n-1}} g$. 
\end{enumerate}
\end{lemma}

\begin{proof}
\begin{enumerate}\item[(i)]  Since $a$ and $b$ both have order two, we never need to write two consecutive $a$'s or $b$'s, so any element in $D(2^n)$ can be written as a product of alternating $a$'s and $b$'s. If an element begins with $b$, then using  the fact that $(ab)^r = (ba)^{2^n - r}$ (see Lemma~\ref{l:ab-order-lemma}) , we can always express the element so that it begins with $a$. 
\item[(ii)] The elements of $D(2^n)$ are either rotations or reflections. Since the rotations are powers of $ab$, it follows that $g=(ab)^r a$ is a reflection, so it has order two. 
\item[(iii)]  This follows from the fact that $(ab)$ has order $2^{n-1}$ and well-known facts about cyclic groups (see  See Theorem 7 in Section 2.3 of~\cite{Dummit-Foote}). 
\item[(iv)] Let $g \in D(2^n)$. By part (i), $g = (ab)^ja^k$. If $k = 0$, then the statement holds since powers of $ab$ commute. If $k = 1$, then collecting powers of $(ab)$ and applying Lemma~\ref{l:ab-order-lemma}, we have 
\begin{align*}
(ab)^{2^{n-1}} g &= (ab)^{2^{n-1}} (ab)^j a \\
&= (ab)^{2^{n-1} +j} a \\
&= (ab)^j (ab)^{2^{n-1}} a  \\
&= (ab)^j a (ba)^{2^{n-1}} \\
&= (ab)^j a (ab)^{2^{n-1}} \\
&= g (ab)^{2^{n-1}}. 
\end{align*}
\end{enumerate}
\end{proof}

Note that the form for elements given in Lemma~\ref{l:normal-form} may not be the shortest or most intuitive form for writing the element. For instance, the generator $b \in D_4$ would be written as $(ab)^3a$ under the convention we've established. Lemma~\ref{l:normal-form} also establishes the following characterization of elements of order two in $D(2^n)$. 

\begin{corollary}\label{c:elements-of-order-two}
Let $G = D(2^n)$. If $g \in G$, then $g$ has order two if and only if $g = (ab)^ja$ for some $0 \leq k \leq 2^n$, or $g = (ab)^{2^{n-1}}$. 
\end{corollary} 

We now turn our attention to describing the subgroups of $D(2^n)$. Although these facts are standard, we will provide proofs that give necessary detail for our later results.  

\begin{proposition}\label{p:dihedral-maximal-subgroups}
For any $n \geq 2$, $D(2^n)$ has three maximal subgroups. Two of these maximal subgroups are isomorphic to $D(2^{n-1})$, and one is isomorphic to $C_{2^n}$.
\end{proposition}

\begin{proof}
We know that the index of the kernel of a homomorphism is equal to the size of its image, and that maximal subgroups in a $p$-group have index $p$. Hence the maximal subgroups of $D(2^n)$ correspond to  the kernels of nontrivial homomorphisms from $D(2^n)$ to $C_2$. Such a homomorphism is completely determined by the images of the generators $a$ and $b$. Let $\alpha_1$ be the homomorphism for which $\alpha_1(a) = e$ and $\alpha_1(b) = \tau$. Let $H_1 = \ker \alpha_1$. It is not hard to see that this group consists of all elements with an even number of $b$'s, and therefore is generated by $a$ and $aba$. Since $a$ and $aba$ both have order two, $H_1$ is isomorphic to a finite dihedral group by Theorem~\ref{l:order-two-dihedral}. We know the order of $H_1$ is $2^n$, so $H_1$ is isomorphic to $D(2^{n-1})$. Let $\alpha_2$ be the homomorphism for which $\alpha_1(b) = e$ and $\alpha_1(a) = \tau$. Letting $H_2 = \ker \alpha_2$, an identical argument to that of $\alpha_1$ establishes that $H_2 \cong D(2^{n-1})$. Finally, let $\alpha_3$ be the homomorphism such that $\alpha_3(a) = \alpha_3(b) = \tau$. Let $H_3 = \ker \alpha_3$. It is clear that $H_3$ contains $ab$, so $H_3$ contains the cyclic subgroup $\langle ab \rangle$, which has order $2^n$ by Lemma~\ref{l:ab-order-lemma}. Since $H_3$ is maximal in $D(2^n)$, we know that $H_3$ has order $2^n$, so $H_3 = \langle ab \rangle \cong C_{2^n}$. 
\end{proof}

\begin{proposition}\label{p:dihedral-or-cyclic}
Every subgroup of a dihedral 2-group $D(2^n)$ is either dihedral or cyclic.
\end{proposition}

\begin{proof}
We will proceed by induction. For the base case $n = 2$, this fact can be explicitly computed for $D(4)$. Now assume the statement is true for some $k \geq 2$, and consider $D(2^{k+1})$. By Proposition~\ref{p:dihedral-maximal-subgroups}, the maximal subgroups of this group are either dihedral or cyclic. Then, by the induction hypothesis and the fact that every proper subgroup is contained in some maximal subgroup, the desired result follows. 
\end{proof}

\begin{corollary}\label{c:powerful-isomorphism-types} 
If a subgroup $H$ of $D(2^n)$ is powerful, then either $H\cong C_2\times C_2$ or $H\cong C_k$ where $k $ is a divisor of $2^{n-1}$. If $H \cong C_k$ for $k > 2$, then $H$ is a subgroup of $\langle ab \rangle$. 
\end{corollary}

\begin{proof}
The first statement follows from Proposition~\ref{p:dihedral-or-cyclic}. The second statement follows from Lemma~\ref{l:normal-form}. 
\end{proof} 

\begin{proposition}\label{p:subgroup-cover-must-contain-maximal-cyclic}
Let $\mathcal{C} = \{H_1, \ldots, H_q \}$ be any subgroup cover of $D(2^n)$. Then there is some $i$ with $1 \leq i \leq q$ such that $H_i = \langle ab \rangle$.
\end{proposition}

\begin{proof}
By definition of subgroup cover, there must be some $i \in \{1,\ldots, q\}$ such that $ab\in H_i$. Since $H_i$ contains $ab$, it follows that $H_ i$ contains the subgroup $\langle ab \rangle$. Since $\langle ab \rangle$ is maximal and $H_i$ must be a proper subgroup of $G$, it follows that $H_i = \langle ab \rangle$. 
\end{proof}

\begin{lemma}\label{l:copies-of-C2-times-C2}
Let $H$ be a subgroup of a dihedral $2$-group. Then $H$ is isomorphic to $C_2 \times C_2$ if and only if $H = \langle (ab)^sa, (ab)^ta \rangle$, where $t \neq s$ and $t + s = 2^{n-1}$.
\end{lemma}

\begin{proof}
Let $t$ and $s$ be distinct positive integers $s + t = 2^{n-1}$. We let $x = (ab)^sa$ and $y = (ab)^ta$, and we let $H$ be a group generated by $x$ and $y$. We want to show that $\langle x , y \rangle \cong C_2 \times C_2$. First, notice that $x$ and $y$ each have order two, by Lemma~\ref{l:normal-form}. We calculate that
\begin{align*}
xy &= (ab)^sa(ab)^ta \\
&= (ab)^s a (ab)^{2^{n-1} - s} a \\
&= (ab)^s a a (ba)^{2^{n-1} - s} \\
&= (ab)^s a a (ab)^{2^{n-1} - s} \\
&= (ab)^s (ab)^{2^{n-1} - s} \\
&= (ab)^{2^{n-1}}.
\end{align*}

A similar calculation establishes that $yx = (ba)^{2^{n-1}}$, and since $(ab)^{2^{n-1}} = (ba)^{2^{n-1}}$ by Lemma~\ref{l:ab-order-lemma}, it follows that $xy = yx$. Hence $\langle x, y \rangle \cong C_2 \times C_2$ by Lemma~\ref{c:C2-C2-fact}. 

For the other direction, we suppose that $H$ is a subgroup of $D(2^n)$ such that $H$ is isomorphic to $C_2 \times C_2$. Note that $H$ must contain an element of the form $(ab)^sa$, otherwise $H$ would be contained in the cyclic group $\langle ab \rangle$. Then $H$ must also contain another element $g$ of order two. We consider two cases for $g$, following Corollary~\ref{c:elements-of-order-two}. 

\textbf{Case 1:} $g = (ab)^{2^{n-1}}$. From Lemma~\ref{l:ab-order-lemma}, we know that $(ab)^s a = a (ab)^{-s}$ and $(ab)^s a$ commutes with $(ab)^{2^{n-1}}$, giving us \begin{align*}
(ab)^{2^{n-1}}(ab)^s a &= (ab)^{2^{n-1}} a (ab)^{-s} \\
&= (ab)^{2^{n-1}} (ab)^{-s} a \\
&= (ab)^{2^{n-1} - s} a. \\
\end{align*}
Letting $t = 2^{n-1} - s$, we have shown that in this case, $H = \langle (ab)^s a, (ab)^t a \rangle$, where $s \neq t$ and $s + t = 2^{n-1}$.

\textbf{Case 2:} $g = (ab)^r a $ for some $r$. Without loss of generality, assume $s > r$. Then $H$ also contains $(ab)^s a g$, which is equal to $(ab)^s a (ab)^r a$, which simplifies to $(ab)^{s-r}$. By Lemma~\ref{l:normal-form}, if $s-r \neq 2^{n-1}$, then $(ab)^{s-r}$ generates a cyclic subgroup of order larger than two, contradicting our assumption that $H \cong C_2 \times C_2$. Thus it must be the case that $s-r = 2^{n-1}$, and the result follows. 
\end{proof}

\begin{corollary}
Let $g \in D(2^n)$ such that $g = (ab)^ja$ for some $0 \leq j \leq 2^{n-1}$. If $H$ is a powerful subgroup of $D(2^n)$ that contains $g$, then $H$ is either trivial, isomorphic to $C_2$, or isomorphic to $C_2 \times C_2$. 
\end{corollary}

\begin{remark}\label{r:count-and-intersection-C2}
From Lemma~\ref{l:copies-of-C2-times-C2}, it follows that there are exactly $2^{n-1}$ subgroups of $D(2^n)$ that are isomorphic to $C_2 \times C_2$ -- one for each pair of numbers $s,t$ with $0 \leq s, t \leq 2^{n-1}$ and $s \neq t$ such that $s+t = 2^n$. It also follows from this Lemma that if $H_i$ and $H_j$ are distinct subgroups isomorphic to $C_2 \times C_2$, then $H_i \cap H_j = \langle (ab)^{2^{n-1}} \rangle $. 
\end{remark}

\subsection{Calculation of The Powerful Covering Numbers}

\begin{proposition}\label{p:upper-bound-covering-number} 
There exists a powerful cover of $D(2^n)$ with $2^{n-1} + 1$ subgroups. 
\end{proposition}

\begin{proof}
Let $H_1 = \langle ab \rangle$, and for $1 \leq r \leq 2^{n-1}$, let $H_{r+1} = \langle (ab)^{r}a, (ab)^{2^n - r}a \rangle$. There are $2^{n-1}+1$ subgroups, each $H_i$ is abelian, and every element of $D(2^n)$ is contained in some $H_i$.  
\end{proof}

\begin{theorem}\label{t:main-result}
The powerful covering number of $D(2^n)$ is $2^{n-1}+1$.
\end{theorem}

\begin{proof}
By Proposition~\ref{p:upper-bound-covering-number}, we know that the powerful covering number of $D(2^n)$ is at most $2^{n-1} + 1$. Now we show that $D(2^n)$ can not be covered by fewer than $2^{n-1} + 1$ powerful subgroups. Let $\mathcal{C} = \{ H_1, \ldots, H_q \}$ be a powerful cover of $D(2^n)$. Appealing to Proposition~\ref{p:subgroup-cover-must-contain-maximal-cyclic} and re-indexing if necessary, we may assume that $H_1 = \langle ab \rangle$, so $|H_1| = 2^n$. Now we claim that for each $i$ with $2 \leq i \leq q$, the subgroup $H_i$ is isomorphic to to either $C_2$ or $C_2 \times C_2$.  This follows from Proposition~\ref{c:powerful-isomorphism-types}, and from the fact that if $H_i \cong C_k$ for some $k > 2$, we would have $H_i \subseteq \langle ab \rangle$. This would make the covering redundant, and hence not minimal. 
From Remark~\ref{r:count-and-intersection-C2} each $H_i$ that is isomorphic to $C_2 \times C_2$ contains $e$ and $(ab)^{2^{n-1}}$, which are already in $H_1$, so each $H_i$ that is isomorphic to $C_2 \times C_2$ contributes two new elements, while each $H_i$ that is isomorphic to $C_2$ contributes one new element.  This means that
\[
\left| \bigcup_{i=1}^q H_i \right| \leq |H_1| + 2(q -1) = 2^{n} + 2(q-1).
\]
In other words, the subgroups $H_2, \ldots, H_q$ can contain at most $2(q-1)$ elements not contained in $H_1$.
If $q < 2^{n-1}$, we would then have \[
\left| \bigcup_{i=1}^q H_i \right| < 2^{n} + 2(2^{n-1}) = 2^{n+1},
\] meaning that this collection of subgroups could not be a cover for the $2^{n+1}$ elements of $D(2^n)$. Thus, any cover of $D(2^n)$ by powerful subgroups must contain at least $2^{n-1} + 1$ powerful subgroups. This completes the proof. 
\end{proof} 

\begin{corollary}
For $n \geq 2$, $\sigma_P(D(2^n)) = \sigma_A(D(2^n))$.
\end{corollary}

\begin{proof}
This follows from Theorem~\ref{t:main-result} and Corollary~\ref{c:powerful-isomorphism-types}.
\end{proof}

\begin{remark}
It is well-known that for any $n \geq 3$, there is a surjective homomorphism from $D(2^{n+1})$ to $D(2^n)$. Theorem~\ref{t:main-result} shows that $\sigma_P(D(2^{n+1})) > \sigma_P(D(2^n))$. This offers an important contrast between $\sigma_P$ and $\sigma$, as it shows that Theorem~\ref{t:homomorphic-cover-numbers} does not adapt to the case of powerful covering numbers. 
\end{remark}

\section{Conclusion}

In this section, we raise questions and provide directions for future work. 

Recall that from the result of Bryce and Serena in Theorem~\ref{t:minimal-covers-by-abelian}, the groups for which $\sigma(G) = \sigma_A(G) = \sigma_P(G)$ have been completely characterized. On the other hand, our results in the previous section establish an infinite family of examples for which $\sigma_A(G) = \sigma_P(G)$, but $\sigma(G) \neq \sigma_P(G)$. It would be interesting to establish the relationship between these properties for further classes of groups.

\textbf{Question 1.} For which $p$-groups is it true that $\sigma_P(G) = \sigma_A(G)$?

If $G$ is a $p$-group of order $p^n$ and nilpotence class $c$, the \textit{coclass} of $G$ is defined as $n-c$. For each value of $n \geq 4$, there are three $2$-groups having order order $2^n$ and  coclass equal to 1: the dihedral group of order $2^n$, the \textit{quasi-dihedral group} of order $2^n$, and the \textit{generalized quartenion group} of order $2^n$.  We have explicitly calculated the powerful covering number for one of them. Investigations by the second and third authors in \texttt{GAP}~\cite{GAP} suggest that the remaining $2$-groups of coclass equal to 1 also have a powerful covering number of $2^{n-1} + 1$, leading us to the following conjecture. 

\textbf{Conjecture 1.} If $G$ is a $2$-group of coclass 1 such that $|G| = 2^{n+1}$, then $\sigma_P(G) = 2^{n-1} + 1$.

Among the examples of $2$-groups that we have considered, this is the largest powerful covering number that we have found, so we also make the following conjecture.

\textbf{Conjecture 2.} If $G$ is a $2$-group of order $2^{n+1}$, then $\sigma_P(G) \leq 2^{n-1} + 1$.

 One would hope that if $H$ is a subgroup of $G$, then $\sigma_P(H) \leq \sigma_P(G)$. The naive attempt at proving this result would proceed as follows. If $K_1, \ldots, K_n$ is a minimal powerful cover of $G$, then $(H \cap K_1), \ldots, (H \cap K_n)$ is a powerful cover of $H$, so $\sigma_P(H) \leq n$. There are two potential problems with this argument: if $H$ is equal to one of the $K_i$, then $H \cap K_i$ is not a proper subgroup. Also, even if $H$ is a powerful subgroup of $G$, it may not be true that $H \cap K_i$ is a powerful subgroup of $H$. 
 
 However, computational investigations by the second and third authors in \texttt{GAP}~\cite{GAP} have searched all $p$-groups of order not exceeding 81, and have not found any examples where a proper subgroup has a larger powerful covering number than the group containing it.  

\textbf{Question 2.} What are the properties of $\sigma_P$ with respect to subgroup inclusion? Is there an example of a group $G$ with a subgroup $H$ such that $\sigma_P(H) > \sigma_P(G)$? 

Another natural extension of this research is to study minimal covers of $p$-groups by \textit{powerfully embedded} subgroups. If $G$ is a finite $p$-group, a normal subgroup $N$ of $G$ is said to be powerfully embedded in $G$ if $[N,G] \subseteq N^p$ when $p$ is odd, or $[N,G] \subseteq N^4$ when $p=2$. Every powerfully embedded subgroup is powerful, yet the converse is not true. Using $\sigma_{PE}(G)$ to denote the size of a minimal powerfully embedded cover, this implies $\sigma_P(G) \leq \sigma_{PE}(G)$. One can show that a powerfully embedded cover is not guaranteed to exist. For example, it is not possible to cover the dihedral group on 32 elements using powerfully embedded subgroups. 

\textbf{Question 3.} For which $p$-groups does a powerfully embedded cover exist? What are the $p$-groups for which $\sigma_P(G) = \sigma_{PE}(G)$? When such a cover exists, what is the relationship between $\sigma_A(G)$ and $\sigma_{PE}(G)$?

We aim to continue work to find answers to the questions posed in this section.

\end{document}